\documentclass[12pt]{article}

\usepackage{amsmath,amssymb,amsbsy,amsfonts,amsthm,latexsym,
                            amsopn,amstext,amsxtra,euscript,amscd}

\begin{document}

\newtheorem{theorem}{Theorem}
\newtheorem{lemma}[theorem]{Lemma}
\newtheorem{claim}[theorem]{Claim}
\newtheorem{cor}[theorem]{Corollary}
\newtheorem{prop}[theorem]{Proposition}
\newtheorem{definition}{Definition}
\newtheorem{question}[theorem]{Open Question}

\def\cA{{\mathcal A}}
\def\cB{{\mathcal B}}
\def\cC{{\mathcal C}}
\def\cD{{\mathcal D}}
\def\cE{{\mathcal E}}
\def\cF{{\mathcal F}}
\def\cG{{\mathcal G}}
\def\cH{{\mathcal H}}
\def\cI{{\mathcal I}}
\def\cJ{{\mathcal J}}
\def\cK{{\mathcal K}}
\def\cL{{\mathcal L}}
\def\cM{{\mathcal M}}
\def\cN{{\mathcal N}}
\def\cO{{\mathcal O}}
\def\cP{{\mathcal P}}
\def\cQ{{\mathcal Q}}
\def\cR{{\mathcal R}}
\def\cS{{\mathcal S}}
\def\cT{{\mathcal T}}
\def\cU{{\mathcal U}}
\def\cV{{\mathcal V}}
\def\cW{{\mathcal W}}
\def\cX{{\mathcal X}}
\def\cY{{\mathcal Y}}
\def\cZ{{\mathcal Z}}

\def\A{{\mathbb A}}
\def\B{{\mathbb B}}
\def\C{{\mathbb C}}
\def\D{{\mathbb D}}
\def\E{{\mathbb E}}
\def\F{{\mathbb F}}
\def\G{{\mathbb G}}
\def\H{{\mathbb H}}
\def\I{{\mathbb I}}
\def\J{{\mathbb J}}
\def\K{{\mathbb K}}
\def\L{{\mathbb L}}
\def\M{{\mathbb M}}
\def\N{{\mathbb N}}
\def\O{{\mathbb O}}
\def\P{{\mathbb P}}
\def\Q{{\mathbb Q}}
\def\R{{\mathbb R}}
\def\S{{\mathbb S}}
\def\T{{\mathbb T}}
\def\U{{\mathbb U}}
\def\V{{\mathbb V}}
\def\W{{\mathbb W}}
\def\X{{\mathbb X}}
\def\Y{{\mathbb Y}}
\def\Z{{\mathbb Z}}

\def\E{{\mathbf E}}
\def\Fp{\F_p}
\def\ep{{\mathbf{e}}_p}

\def\scr{\scriptstyle}
\def\\{\cr}
\def\({\left(}
\def\){\right)}
\def\[{\left[}
\def\]{\right]}
\def\<{\langle}
\def\>{\rangle}
\def\fl#1{\left\lfloor#1\right\rfloor}
\def\rf#1{\left\lceil#1\right\rceil}
\def\le{\leqslant}
\def\ge{\geqslant}
\def\eps{\varepsilon}
\def\mand{\qquad\mbox{and}\qquad}

\def\li{\mathrm{li}}

\newcommand{\comm}[1]{\marginpar{%
\vskip-\baselineskip 
\raggedright\footnotesize
\itshape\hrule\smallskip#1\par\smallskip\hrule}}

\def\xxx{\vskip5pt\hrule\vskip5pt}


\title{\bf On the Lang--Trotter and Sato--Tate Conjectures
on Average for Polynomial Families of  Elliptic Curves}

\author{ {\sc Igor E. Shparlinski} \\
{Department of Computing, Macquarie University} \\
{Sydney, NSW 2109, Australia} \\
{\tt igor.shparlinski@mq.edu.au}}

\date{\today}
\pagenumbering{arabic}

\maketitle

\begin{abstract}
We show that  the reductions modulo primes $p\le x$ of the elliptic curve
$$
 Y^2 = X^3 + f(a)X + g(b),  
$$
behave as predicted by the Lang--Trotter and Sato--Tate conjectures, on average  over
integers $a \in [-A,A]$ and $b \in [-B,B]$ for $A$ and $B$ reasonably 
small compared to $x$, 
provided that  $f(T), g(T) \in \Z[T]$ are not powers of another polynomial over $\Q$. 
For $f(T) = g(T) = T$ first results of this kind are due 
to \'E.~Fouvry and  M.~R.~Murty
 and have been further extended  by  other authors. Our technique 
is different from that of \'E.~Fouvry and M.~R.~Murty which 
does not seem to work in the case of general polynomials $f$ and $g$. 
\end{abstract}

\section{Introduction}
\label{sec:intro}

\subsection{Background} 

For a prime $p$, we denote by $\Fp$ the finite field with $p$
elements.

Given an elliptic curve $\E$ over $\Q$ and a prime 
$p$ we use 
$\E(\Fp)$ to denote the set  of
the $\Fp$-rational points of the reduction of $\E$
modulo $p$, provided that $p$ does not divide the
discriminant $\Delta(\E)$ of $\E$,  together with a point at
infinity. This set forms an \emph{abelian group} under an appropriate
composition rule and satisfies the \emph{Hasse
bound}:
\begin{equation}
\label{eq:HW bound} 
\left|\#\E(\F_p)- p - 1\right|\le 2\sqrt{p},
\end{equation}
see~\cite{Silv} for a background on elliptic curves.

Accordingly, we denote by $\Pi^{\tt LT}(\E,t;x)$ the number of
primes $3 < p \le x$ (with $p \nmid \Delta(\E)$) for which 
$ \#\E(\F_p)= p + 1- t$. The \emph{Lang--Trotter  conjecture}
asserts that if $\E$ does not have complex multiplication,
then the asymptotic formula
\begin{equation}
\label{eq:LT conj} \Pi^{\tt LT}(\E,t;x) \sim
 c(\E,t)\frac{\sqrt{x}}{\log x},\qquad x\to\infty,
\end{equation}
holds for some constant $c(\E,t) \ge 0$ depending only on $\E$ and $t$.

Since the Lang--Trotter conjecture~\eqref{eq:LT conj} remains widely open, 
see~\cite{Elk1,Elk2,Elk3,MurMurSar,Murt-VK2,Murt-VK3,MurSch,Ser,Wan}, 
it is natural to obtain its analogues  ``on average'' over various interesting 
families of curves.
For example, for integers $a$ and $b$ such that $4a^3+27b^2\ne 0$, we denote by
$\E_{a,b}$ the elliptic curve defined by the \emph{affine
Weierstra\ss\ equation}:
$$
\E_{a,b}~:~Y^2 = X^3 + aX + b.
$$
and put 
$$
\Pi^{\tt LT}_{a,b}(t;x) = \Pi^{\tt LT}(\E_{a,b},t;x).
$$

Fouvry and Murty~\cite{FoMu} have initiated the study of
$\Pi^{\tt LT}_{a,b}(t;x) $
and similar quantities on average and shown that 
 the asymptotic formula
\begin{equation}
\label{eq:L-T Aver}
\frac{1}{4AB}\sum_{|a|\le A}\sum_{|b|\le
B}\Pi^{\tt LT}_{a,b}(t;x)\sim C(t)\frac{\sqrt x}{\log
x},\qquad x\to\infty, 
\end{equation}
holds for $t=0$ (with $C(0) = \pi/3$) in the range
\begin{equation}
\label{eq:FM threshold} AB\ge x^{3/2 + \eps} \mand \min\{A,B\} \ge
x^{1/2 + \eps},
\end{equation}
for arbitrary fixed $\eps>0$
(for $4a^3+27b^2= 0$ we define  $\Pi^{\tt LT}_{a,b}(t;x) = 0$).  
David and Pappalardi~\cite{DavPapp1} have proved that~\eqref{eq:L-T Aver}
holds for any fixed $t$, with  some explicit constant $C(t)>0$ 
depending only on $t$,
but in a smaller range than~\eqref{eq:FM threshold}. This range
has been expanded to the original level  described by
the conditions~\eqref{eq:FM threshold} by Baier~\cite{Baier1}.
Finally, Baier~\cite{Baier2} has obtained~\eqref{eq:L-T Aver} 
in an even wider range
$$
AB\ge x^{3/2 + \eps} \mand \min\{A,B\} \ge
x^{\eps}
$$
(in this range  the  term with $ab=0$ must be eliminated from 
the summation in~\eqref{eq:L-T Aver} as their contribution 
may exceed the main term).

Furthermore, for an elliptic curve $\E$ over $\Q$ 
and a prime $p > 3$ 
we recall~\eqref{eq:HW bound} and 
define the angle $\psi (\E; p) \in [0, \pi]$ via the
identity
\begin{equation}
\label{eq:ST angle}
p+1 - \#\E(\F_p)  =  2\sqrt{p}\, \cos\psi(\E; p) .
\end{equation}
(we also define $\psi (\E; p)$ arbitrary, say as $\psi (\E;p) =0$ 
if  $p\mid 4a^3+27b^2$). 
For  $0 \le \alpha < \beta \le \pi$, 
we denote by $\Pi^{\tt ST}(\E,\alpha,\beta;x)$ the number of
primes $p \le x$ (with $p \nmid 4a^3+27b^2$) for which $ \alpha
\le \psi (\E;p) \le \beta$. Then, the  \emph{Sato--Tate  conjecture}
asserts that if $\E$ does not have complex multiplication,
then the asymptotic formula
\begin{equation}
\label{eq:ST conj} \Pi^{\tt ST} (\E,\alpha,\beta;x) \sim
\mu_{\tt ST}(\alpha,\beta)  \frac{x}{\log x},\qquad x\to\infty,
\end{equation}
holds, where 
\begin{equation}
\label{eq:ST dens}
\mu_{\tt ST}(\alpha,\beta) = \frac{2}{\pi}\int_\alpha^\beta
\sin^2\gamma\, d \gamma,
\end{equation}
is the \emph{Sato--Tate  density}, see~\cite{Birch,Katz,Murt-VK1}).

 The method of Fouvry and Murty~\cite{FoMu} is based on 
bounds of exponential sums; it is quite universal and 
has been applied to a number of related questions, 
see~\cite{AkDavJur,BaZh,BaCoDa,BBIJ,DavPapp1,DavPapp2,James,JamYu,Shp1,Shp2}.  
In particular, using this method, it is easy to show that 
\begin{equation}
\label{eq:S-T Aver}
\frac{1}{4AB}\sum_{|a|\le A}\sum_{|b|\le
B}\Pi^{\tt ST}_{a,b}(t;x)\sim\mu_{\tt ST}(\alpha,\beta)  \frac{x}{\log x},
\qquad x\to\infty, 
\end{equation}
in the same range~\eqref{eq:FM threshold}, where as before we define $$
\Pi^{\tt ST}_{a,b}(t;x) = \Pi^{\tt ST}(\E_{a,b},t;x).
$$

Note that Taylor~\cite{Tayl} has recently given a complete 
proof of~\eqref{eq:ST conj} (except for the curves $\E$ 
with integral $j$-invariant), but this  does not imply any
results on average due to the lack of uniformity with respect to
the coefficients $a$ and $b$ in the Weierstra\ss\ equation.

Recently, a different approach has been suggested in~\cite{BaSh}, which 
is based on bounds of multiplicative character sums, 
see also~\cite{BaCoDa,Shp2} for further application of this 
approach. We remark that the original purpose of~\cite{BaSh} has 
been to improve some results on the {\it Sato--Tate\/} conjecture 
for the curves $\E_{a,b}$ on average, and obtain~\eqref{eq:S-T Aver}
in a wider range than~\eqref{eq:FM threshold}, namely 
under the conditions
\begin{equation}
\label{eq:BS threshold} AB\ge x^{1 + \eps} \mand \min\{A,B\} \ge
x^{\eps},
\end{equation}
see also~\cite{Baier2,BaZh} for some other approaches.  

\subsection{Our Results}  

Here we show that the ideas of~\cite{BaSh} can 
also be used to  
extend~\eqref{eq:L-T Aver} and~\eqref{eq:S-T Aver}
to more general families of curves. 

Let us fix two polynomials $f(T), g(T) \in \Z[T]$ that  
 are not powers of another polynomial over $\Q$.

Here we consider the family of curves $\E_{f(a),g(b)}$
and obtain analogues of the asymptotic 
formulas~\eqref{eq:L-T Aver} and \eqref{eq:S-T Aver} for these curves, 
that is, we obtain asymptotic formulas for the average values
$$
\frac{1}{4AB}\sum_{|a|\le A}\sum_{|b|\le B}
\Pi^{\tt LT}_{f(a),g(b)}(t;x) \mand \frac{1}{4AB} \sum_{|a|\le A}\sum_{|b| \le B}
\Pi^{\tt ST}_{f(a),g(b)}(\alpha,\beta;x).
$$
As usual for questions of this kind, our main concern is to 
keep the amount of averaging as little as possible and thus
obtain these asymptotic formulas in the ranges comparable with
those given  by~\eqref{eq:FM threshold} and~\eqref{eq:BS threshold}. 

We note that the 
approach of  Fouvry and Murty~\cite{FoMu} does not 
seem to work for the families of curves $\E_{f(a),g(b)}$. 

We also study a one-parametric family of curves $\E_{f(a), g(a)}$
with two polynomials $f$ and $g$ 
over $\F_p$ (satisfying some natural condition). 
Using a result of Michel~\cite{Mich} we 
show that for a fixed $\eps>0$ and a sufficiently large prime $p$, 
the   corresponding angles $\psi(\E_{f(a), g(a)},p)$   are distributed 
with the Sato--Tate density when $a$ runs through consecutive integers
of  an interval of length at least $p^{3/4+\eps}$.   

Finally, we recall that upper bounds for $\Pi^{\tt LT}_{f(\rho),g(\rho)}(t;x)$,  on average 
when $\rho$ runs through the set of Farey fractions of order $Q$, 
are given in~\cite{CojHal,CojShp}. 

\subsection{Notation}  
   
Throughout the paper, any implied constants in the symbols $O$ and
$\ll$ may occasionally depend, where obvious, on the polynomials
$f$ and $g$ and the real parameters
$\eps$,  but are absolute otherwise. We recall that the
notations $U \ll V$ and  $U = O(V)$ are both equivalent to the
statement that the inequality $|U| \le c\,V$ holds with some
constant $c> 0$.

The letters $p$ and $q$ always denote prime numbers, while $m$ and
$n$ always denote integers. As usual, we use $\pi(x)$ to denote
the number of primes $p\le x$.

\section{Character Sums and Distribution of Power Residues}
\label{sec:Char Res}

\subsection{Character Sums}

For a prime $p$, we denote by $\cX_p$ the set of multiplicative
characters of $\Fp$, $\chi_0$ the principal character of $\Fp$,
and $\cX_p^*=\cX_p\setminus\{\chi_0\}$ the set of nonprincipal
characters; we refer the reader to~\cite[Chapter~3]{IwKow} for the
necessary background on multiplicative characters. We recall the
following orthogonality relations. For any integer $f \mid p-1$ and
$v \in \F_p$, 
\begin{equation}
\label{eqn:orth chi/u}
 \frac{1}{f} \sum_{\substack{\chi\in \cX_p\\
\chi^{f}=\chi_0}}\chi(v) =
\left\{\begin{array}{ll}
1&\quad\text{if $v= w^f$ for some   $w \in \F_p^*$,}\\
0&\quad\text{otherwise.}
\end{array}\right.
\end{equation}
Also, we have
\begin{equation}
\label{eqn:orth chi_1/chi_2} \frac{1}{p-1} \sum_{u\in
\F_p^*}\chi_1(u) \overline\chi_2(u) = \left\{\begin{array}{ll}
1&\quad\text{if $\chi_1=\chi_2$,}\\
0&\quad\text{otherwise,}
\end{array}\right.
\end{equation}
for all  $\chi_1,\chi_2 \in\cX_p$ (here,
$\overline\chi_2$ is the character obtained from $\chi_2$ by
complex conjugation).

The following result is a special case of the Weil bound 
(see~\cite[Equations~(12.23)]{IwKow})

\begin{lemma}
\label{lem:Char Compl} For any prime $p$ and  polynomial $h(T) \in \Z[T]$ 
which is not a power of another polynomial modulo $p$, uniformly 
over all  integers  $m$ and nontrivial multiplicative characters $\chi$ modulo 
$p$, we have
$$
\sum_{u= 1}^{p} \chi(h(u)) \exp\(2\pi i
\frac{mu}{  p}\) \ll p^{1/2} .
$$
\end{lemma}

Combining Lemma~\ref{lem:Char Compl} 
with the standard reduction between complete and incomplete sums
(see~\cite[Section~12.2]{IwKow}), we obtain the following result.

\begin{lemma}
\label{lem:Char Incompl} For any prime $p$ and  polynomial $h(T) \in \Z[T]$ 
which is not a power of another polynomial modulo $p$,  uniformly over all positive 
integers $L,M$ and nontrivial multiplicative characters $\chi$ modulo 
$p$,  we have
$$
\sum_{n=L+1}^{L+M} \chi(h(n)) \ll  (M/p + 1)p^{1/2} \log p.
$$
\end{lemma}

\subsection{Distribution of Powers}

Adapting  the idea of Fouvry and Murty~\cite{FoMu}
we study  the distribution of the pairs
\begin{equation}
\label{eq:pairs} \{(ru^4, su^6)~:~u\in \F_p\},
\end{equation}
where $r,s\in \F_p^*$, 
among the residues modulo $p$ of the polynomial values $(f(a), g(b))$
with  $|a| \le A$, $|b|\le B$. 
However, instead of exponential sums, used in~\cite{FoMu}, 
we follow the approach of~\cite{BaSh} and study the distribution of 
the pairs~\eqref{eq:pairs} using multiplicative character sums.

We also note that since the polynomials $f(T), g(T) \in \Z[T]$ that  
are not powers of another polynomial over $\Q$, for a sufficiently large 
prime $p$ they are not powers of a polynomial modulo $p$. Thus
Lemma~\ref{lem:Char Incompl} applies to character sums with $f(T)$ and $g(T)$.

We begin by investigating the distribution of
the second component $su^6$ of 
the pairs~\eqref{eq:pairs}. Accordingly, we
define
$$
\cZ_s(B;p)=\{(u,b) \in\Fp^*\times [-B,B]~:~s u^6 \equiv g(b) \pmod p,  \ 
|b| \le B\}.
$$
We note that although we are usually only interested in the 
first component $u$, we defined $\cZ_s(B;p)$ as a set of pairs 
$(u,b)$, which essentially means that each $u$ is taken 
with the mulitplicity of the the residue class $s u^6$ 
amongst the elements of $[-B,B]$.

We have the following bound on the cardinality of $\cZ_s(B;p)$:

\begin{lemma}
\label{lem:ZsB} For any prime $p$, integer $B\ge 1$
 and $s \in \F_p^*$, we have
$$
\#\cZ_s(B;p) = 2B + O\((B/p+1)p^{1/2 + o(1)}\)
$$
as $p\to \infty$. 
\end{lemma}

\begin{proof}
We now define 
$$
d_p = \gcd(p-1,6).
$$
By the orthogonality relation~\eqref{eqn:orth chi/u}, 
for all $n \in\Z$ we have
$$
\#\{u\in\Fp^*~:~u^6\equiv n\pmod p\}=\sum_{\substack{\chi\in
\cX_p\\ \chi^{d_p}=\chi_0}}\chi(n).
$$
If $\overline s$ is an integer such that $s\overline s\equiv
1\pmod p$, it follows that
$$
\#\cZ_s(B;p)=\sum_{|b| \le B}\sum_{\substack{\chi\in \cX_p\\
\chi^{d_p}=\chi_0}}\chi\(\overline s g(b)\)
= 2B + O(1) +\sum_{\substack{\chi\in \cX_p^*\\ \chi^{d_p}=\chi_0}}
\overline \chi(s) \sum_{|b| \le B}\chi\(g(b)\).
$$
Using Lemma~\ref{lem:Char Incompl} we conclude the proof.
\end{proof}

We now take into account the distribution of the first component
$ru^4$ of
the pairs~\eqref{eq:pairs}. For any integers $A,B\ge 1$ and $r,s \in \F_p$, 
we define the set of triples:
\begin{equation*}
\begin{split}
\cZ_{r,s}(A,B;p)=\{(u,a,b)~:~r u^4 \equiv & f(a) \pmod p, \\
& (u,b)\in \cZ_s(B;p),\ |a|\le A \}.
\end{split}
\end{equation*}

\begin{lemma}
\label{lem:ZrsAB} For  any  prime $p$,   
 integers $A,B \ge 1$   and $s \in \F_p^*$, we have
$$
\sum_{r\in\Fp^*}\left|\#\cZ_{r,s}(A,B;p) -
 \frac{2A\cZ_s(B;p)}{p-1}\right|^2 \le 
\(\frac{A}{p}+1\)^2 \(\frac{B}{p} + 1\) Bp^{1+ o(1)}
$$
as $p\to\infty$.
\end{lemma}

\begin{proof} 
Using~\eqref{eqn:orth chi/u} it
follows that
\begin{equation*}
\begin{split}
\#\cZ_{r,s}&(A,B;p)=\sum_{(u,b)\in \cZ_s(B;p) } \sum_{ |a|\le A } 
\frac{1}{p-1} \sum_{\chi\in\cX_p}\chi(ru^4) \overline{\chi\(f(a)\)}\\
&=\frac{2 A \# \cZ_s(B;p)  }{p-1} + 
\frac{1}{p-1}\sum_{\chi\in\cX_p^*} \chi(r) \sum_{(u,b)\in \cZ_s(B;p)
}\chi(u^4) \sum_{ |a|\le A } \overline\chi\(f(a)\).
\end{split}
\end{equation*}
Thus
\begin{equation*}
\begin{split}
\#\cZ_{r,s}&(A,B;p)  -
\frac{2 A \# \cZ_s(B;p)  }{p-1} \\
\ll &
\frac{1}{p-1} \left|\,\sum_{\chi\in\cX_p^*}
\chi(r) \sum_{(u,b)\in \cZ_s(B;p) }\chi(u^4) \sum_{ |a|\le A } 
\overline\chi\(f(a)\) \right|. \end{split}
\end{equation*}

Hence, 
\begin{equation}
\label{eq:Mult Sum} 
\sum_{r\in\Fp^*}\left|\#\cZ_{r,s}(A,B;p)  -
\frac{4 A B}{p-1} \right|^2 \ll  W,
\end{equation}
where
$$
W= \frac{1}{(p-1)^2} \sum_{r\in\Fp^*} \left|\,\sum_{\chi\in\cX_p^*}
\chi(r) \sum_{(u,b)\in \cZ_s(B;p) }\chi(u^4) \sum_{ |a|\le A } 
\overline\chi\(f(a)\) \right|^2.
$$
Squaring out and changing the order of summation, we derive
\begin{equation*}
\begin{split}
W =  \frac{1}{(p-1)^2}& \sum_{\chi_1, \chi_2\in\cX_p^*} \sum_{(u_1,b_1),(u_2,b_2)\in
\cZ_s(B;p)} \chi_1(u_1^4) \overline\chi_2(u_2^4) \\
& \qquad\qquad \sum_{|a_1|,|a_2|\le A } \overline\chi_1\(f(a_1)\) 
\chi_2\(f(a_2)\)\sum_{r\in\Fp}
\chi_1(r) \overline\chi_2(r).
\end{split}
\end{equation*}
Using the orthogonality relation~\eqref{eqn:orth chi_1/chi_2} we
deduce that
$$ W =  \frac{1}{p-1}  \sum_{\chi\in\cX_p^*} \left|\sum_{(u,b)\in
\cZ_s(B;p)} \chi(u^4) \right|^2
    \left|\sum_{ |a|\le A }  \chi\(f(a)\)\right|^2.
$$
Using Lemma~\ref{lem:Char Incompl}, it follows that
\begin{equation}
\label{eq:W2}
W \le (A/p+1)^2p^{o(1)}\sum_{\chi\in\cX_p^*}
 \left|\sum_{(u,b)\in \cZ_s(B;p)} \chi(u^4) \right|^2.
\end{equation}

We now extend the summation in~\eqref{eq:W2} to include the
trivial character $\chi=\chi_0$. Then by the orthogonality 
relation~\eqref{eqn:orth chi/u},   we have
\begin{equation}
\label{eq:W2T}
\sum_{\chi\in\cX_p^*} \left|\sum_{(u,b)\in \cZ_s(B;p)} \chi(u^4)
\right|^2   \le   \sum_{\chi\in\cX_p} \left|\sum_{(u,b)\in \cZ_s(B;p)}
\chi(u^4) \right|^2  = (p-1) T,
\end{equation}
where $T$ is the number of solutions to the congruence
$$
u_1^4\equiv u_2^4 \pmod p,\qquad (u_1,b_1),(u_2,b_2)\in
\cZ_s(B;p).
$$
Obviously $T$ does not exceed  
\begin{equation}
\label{eq:12th powers}
u_1^{12}\equiv u_2^{12} \pmod p,\qquad (u_1,b_1),(u_2,b_2)\in
\cZ_s(B;p).
\end{equation}
Since $su_j^6 \equiv g(b_j)\pmod p$ for some $b_j$ with $|b_j| \le B$,
$j = 1, 2$, 
each solution to~\eqref{eq:12th powers} leads to a congruence
$$
g(b_1)^{2} \equiv g(b_2)^{2}\pmod p,
\qquad |b_1|,|b_2| \le B.
$$
Therefore $T\ll B(B/p+1)$ (because
each $b_j$ corresponds to at most six values of $u_j$).

Thus recalling~\eqref{eq:W2} and~\eqref{eq:W2T}, 
we  conclude the proof.
\end{proof}

\section{Elliptic  Curves}
 
\subsection{Isomorphic Elliptic  Curves}
\label{sec:Isom}

It is well known that if $a,b,r,s\in\Fp$, then $\E_{a,b}(\F_p)\cong\E_{r,s}(\F_p)$,
that is, the curves
$\E_{a,b}$ and $\E_{r,s}$ are \emph{isomorphic over $\F_p$}, if and
only if
\begin{equation}
\label{eq:isom cond}
 a=ru^4 \mand b=su^6
\end{equation}
for some $u \in \F_p^*$. In
particular, each curve $\E_{a,b}$ with $a,b\in \F_p^*$ is
isomorphic to $(p-1)/2$ elliptic curves $\E_{r,s}$, and there are
$2p + O(1)$ distinct isomorphism classes of elliptic curves over
$\Fp$; see~\cite{Len}. 

Therefore we see that  
a link between the distribution of ellipic curves 
of various types and  sets $\cZ_{r,s}(A,B;p)$ is
given by the following trivial statement. 

For an arbitrary set $\cS \subseteq \Fp\times\Fp$, we denote by
$M_p(\cS,A,B)$ the number of curves $\E_{f(a),g(b)}$ such that the
reduction modulo $p$ of the pair $(f(a),g(b))$ belongs to $\cS$,  
$a \in  [-A,A]$ and $b \in [-B,B]$.  

\begin{lemma}
\label{lem:S and Z} Suppose 
that $f(T), g(T) \in \Z[T]$ are not powers of another polynomial 
over $\Q$.  Assume that for a prime $p>3$ we are given
a sets $\cS \subseteq \Fp^*\times\Fp^*$ such that
whenever $(r,s)\in\cS$ and $\E_{a,b}(\F_p)\cong\E_{r,s}(\F_p)$ it
follows that $(a,b)\in\cS$.
Then for any integers $A,B \ge 1$,  the following bound holds:
$$
M_p(\cS,A,B)= \frac{1}{p-1} \sum_{(r,s)\in
\cS}\#\cZ_{r,s}(A,B;p)+O\(AB/p + A + B\).
$$
\end{lemma}

\begin{proof} We estimate the contribution from the 
curves with 
$$
f(a)g(b)\(4f(a)^3+27g(b)^2\) \equiv 0 \pmod p
$$
trivially as 
$$O\(\(A/p+1\)B + A\(B/p+1)\)\)  = O\(AB/p + A + B\).
$$
We also note that if $a \equiv ru^4 \pmod p$ and 
$b\equiv su^6 \pmod p$  then each group
$\E_{f(a),g(b)}(\Fp)$ with $|a|\le A$ and $|b|\le B$ is counted
precisely $p-1$ times in the sum on the right-hand side.
\end{proof}

\subsection{Statistics of Elliptic Curves}
\label{sec:Ell Curves}

Let $\cR_p(t)$ be be the set of pairs $(r,s)\in
\F_p^*\times\F_p^*$ such that
$$
\#\E_{r,s}\(\F_p\)  = p+1 - t.
$$

We recall the following well know estimate, see, for example,~\cite[Proposition~1.9]{Len}.

\begin{lemma}
\label{lem:LT Upper} For any fixed $t$,
$$
\cR_p(t) \ll p^{3/2 + o(1)}. 
$$
\end{lemma}

We also define 
$$
\li_{1/2}(x) = \int_{2}^x \frac{d\, z}{2 z^{1/2} \log z} = \(1 + o(1)\) \frac{x^{1/2}}{\log x}.
$$

By a result of David and Pappalardi~\cite[Equations~(24) and (29)]{DavPapp1}
(see also~\cite[Equations~(2.1) and Lemma~3]{Baier1}) we have

\begin{lemma}
\label{lem:LT Stat} For any fixed integer $t$, there exists a constant $C(t)> 0$ such that 
for any fixed $C > 0$ 
$$
\sum_{p \le x} \frac{1}{p^2} \# \cR_p(t)  
= C(t)\li_{1/2}(x) + O\(x^{1/2} (\log x)^{-C}\).
$$
\end{lemma}

Let $\cT_p(\alpha,\beta)$ be the set of pairs $(r,s)\in
\F_p^*\times\F_p^*$ such that the inequalities $ \alpha \le
\psi_{r,s}(p) \le \beta$ hold, where the angles 
$$\psi_{r,s}(p) = \psi(\E_{r,s};p)
$$
are given by~\eqref{eq:ST angle}. 
It is natural to expect  that
$$
\# \cT_p(\alpha,\beta) \sim \mu_{\tt ST}(\alpha,\beta) p^2
$$
as $p\to\infty$, 
where $\mu_{\tt ST}(\alpha,\beta)$ is given by~\eqref{eq:ST dens}, 
which is known as the Sato--Tate conjecture in \emph{the vertical aspect}.
It has been established  by  Birch~\cite{Birch} (see also~\cite{MilMur}),
but here we require a stronger result What is needed is a full
analogue for the Sato--Tate density of the bound of
Niederreiter~\cite{Nied} on the discrepancy of the distribution of
values of  Kloosterman sums. Fortunately, such a
result can be obtained using the same methods since all of the
underlying tools, namely~\cite[Lemma~3]{Nied}
and~\cite[Theorem~13.5.3]{Katz}, apply to $\psi_{r,s}(p) $  as
well as to values of Kloosterman sums. In particular,
from~\cite[Theorem~13.5.3]{Katz} it follows that
\begin{equation}
\label{eq:Katz bound}
\frac{1}{(p-1)^2} \sum_{\substack{r,s\in \F_p^*\\4r^3+27s^2 \ne
0}} \frac{\sin\((n+1)\psi_{r,s}(p)\)}{\sin\( \psi_{r,s}(p)\)} \ll
np^{-1/2},  \qquad n =1,2, \ldots\,,
\end{equation}
(see also the work of Fisher~\cite[Section~5]{Fish}). Thus, as
in~\cite{Nied}, we have:

\begin{lemma}
\label{lem:ST Stat} For any prime $p$, we have
$$
\max_{0 \le \alpha < \beta \le \pi} \left|\# \cT_p(\alpha,\beta) -
\mu_{\tt ST}(\alpha,\beta) p^2 \right| \ll p^{7/4}.
$$
\end{lemma}

Michel~\cite[Proposition~1.1]{Mich} gives a version 
of~\eqref{eq:Katz bound} for one parametric polynomial families of 
curves, where the sums is also twisted by additive characters. 

\begin{lemma}
\label{lem:Mich bound} For any prime  $p$ and uniformly over  all 
integers $m$, for any polynomials
$f(T),g(T) \in \Z[T]$,   we have
$$
\frac{1}{p} \sum_{\substack{a \in \F_p\\
4f(a)^3 +27g^2(a)\not \equiv 0 \pmod p}} 
\frac{\sin\((n+1)\psi_{f(a), g(a)}(p)\)}{\sin\( \psi_{f(a), g(a)}(p)\)}  \exp\(2\pi i
\frac{m u}{  p}\) \ll
np^{-1/2},  
$$
for $n =1,2, \ldots$.
\end{lemma}

Again, using  the standard reduction between complete and incomplete sums
(see~\cite[Section~12.2]{IwKow}),  we see that Lemma~\ref{lem:Mich bound}
implies the following result.

\begin{lemma}
\label{lem:Mich bound Incomp} For any prime  $p$,  integer $A\ge 1$ and  polynomials
$f(T),g(T) \in \Z[T]$,   we have
$$
\frac{1}{p} \sum_{\substack{|a| \le A \\
4f(a)^3 +27g^2(a)\not \equiv 0 \pmod p}} 
\frac{\sin\((n+1)\psi_{f(a), g(a)}(p)\)}{\sin\( \psi_{f(a), g(a)}(p)\)}  \ll
np^{-1/2},  
$$
for $n =1,2, \ldots$. 
\end{lemma}

Let $\cT_{f,g,p}(A;\alpha,\beta)$ be the set of integers $a$ 
with $|a| \le A$ and  
such that the inequalities $ \alpha \le
\psi_{f(a), g(a)}(p) \le \beta$ hold, where, as before, the angles 
$$\psi_{f(a), g(a)}(p) = \psi(\E_{f(a), g(a)};p)
$$
are given by~\eqref{eq:ST angle}. 

Applying the technique of Niederreiter~\cite{Nied}  we instantly obtain 
the following analogue of Lemma~\ref{lem:ST Stat}. 

\begin{lemma}
\label{lem:ST Stat Poly} For any prime  $p$, positive integer $A < p/2$  and  polynomials
$f(T),g(T) \in Z[T]$ such that $4f(T)^3 +27g(T)^2$ is not identical 
to zero,   we have 
$$
\max_{0 \le \alpha < \beta \le \pi} \left|\# \cT_{f,g,p}(A;\alpha,\beta) -
2\mu_{\tt ST}(\alpha,\beta) A \right| \ll A^{1/2}p^{1/4}.
$$
\end{lemma}

\begin{proof} By~\cite[Lemma~3]{Nied} we see that for any integer $k$
\begin{equation*}
\begin{split}
\max_{0 \le \alpha < \beta \le \pi} &\left|\# \cT_{f,g,p}(A;\alpha,\beta) -
2\mu_{\tt ST}(\alpha,\beta) A \right| \\
&\ll   \frac{A}{k} +   \sum_{n=1}^k\frac{1}{n}
 \left|\sum_{\substack{|a| \le A \\
4f(a)^3 +27g^2(a)\not \equiv 0 \pmod p}} 
\frac{\sin\((n+1)\psi_{f(a), g(a)}(p)\)}{\sin\( \psi_{f(a), g(a)}(p)\)} \right|.
\end{split}
\end{equation*}
Applying Lemma~\ref{lem:Mich bound Incomp} and choosing 
$k = \rf{A^{1/2}p^{-1/4}}$, we obtain the desired bound. 
\end{proof}

We note that Lemma~\ref{lem:ST Stat Poly}
is a generalisation of~\cite[Theorem~1.4]{MilMur}
that corresponds to the case of $A = (p-1)/2$ that 
follows directly from Lemma~\ref{lem:Mich bound}
applied with $m = 0$.

\section{Main Results}
\label{sec:main}

\subsection{General Estimate}

We now have the following general results which can be applied
to various families  of elliptic curves.  
We formulate it in a more general form that we need for applications
to the Lang-Trotter or Sato--Tate conjectures. 

Let us define
\begin{equation*}
\begin{split}
E_\vartheta(U,V;z) =  UVz^{-1/2 -\vartheta/2 }  +  
UV^{1/2} z^{-\vartheta/2} + UV&z^{-1}+ U z^{1/2 - \vartheta} +
U \\
  +&   Vz^{1/2 -\vartheta/2}    + V^{1/2} z^{1-\vartheta/2}.
\end{split}
\end{equation*} 

\begin{theorem}
\label{thm:M and S} Suppose 
that $f(T), g(T) \in \Z[T]$ are not powers of another polynomial 
over $\Q$.  Assume that for a prime $p>3$ we are given
a set $\cS \subseteq \Fp^*\times\Fp^*$ of cardinality
$$
\# \cS \le p^{2-\vartheta+ o(1)}
$$
as $p\to \infty$, for some absolute constant $\vartheta \ge 0$,  and  such that
whenever $(r,s)\in\cS$ and $\E_{a,b}(\F_p)\cong\E_{r,s}(\F_p)$ it
follows that $(a,b)\in\cS$. 
 Then for any integers $A,B \ge 1$,  the following bound holds:
$$ \left|M_p(\cS,A,B)    - \frac{ 4 A B \#\cS  }{(p-1)^2} \right|
\le E_\vartheta(U,V;p) p^{o(1)}, 
$$
where 
$$
U = \max\{A,B\} \mand V = \min\{A,B\}.
$$
\end{theorem}

\begin{proof}  Assume that $A \ge B$. 
Using Lemma~\ref{lem:S and Z}, we derive
\begin{equation*}
\begin{split}
M_p(\cS,A, B)  & - \frac{4 A B \#\cS }{(p-1)^2}\\
 & \ll~ \frac{1}{p-1}\sum_{(r,s)\in
\cS}\left|\#\cZ_{r,s}(A,B;p)-  \frac{4 A B }{p-1}\right| +\frac{A B }{p} + A + B.
\end{split}
\end{equation*}
Furthermore, by Lemma~\ref{lem:ZsB} (and since $A \ge B$) we see that 
\begin{equation}
\label{eq:M Delta}
\begin{split}
M_p(\cS,A, B)  - &\frac{4 A B \#\cS }{(p-1)^2} \\ 
 & \ll \Delta  + ABp^{-1/2-\vartheta+o(1)}+
A p^{1/2-\vartheta+o(1)}  + ABp^{-1}  + A, 
\end{split} 
\end{equation}
where
$$
\Delta = \frac{1}{p-1}\sum_{(r,s)\in\cS}
\left|\#\cZ_{r,s}(A,B;p)-  \frac{2 A \# \cZ_s(B,p) }{p-1}\right|. 
$$

By the Cauchy inequality, 
\begin{eqnarray*}
\Delta^2 & \ll &  \frac{\#\cS }{p^2}
\sum_{(r,s)\in\cS}
\left|\#\cZ_{r,s}(A,B;p)-  \frac{2 A \# \cZ_s(B,p) }{p-1}\right|^2\\
& \le &  p^{-\vartheta+o(1)}
\sum_{(r,s)\in\cS}
\left|\#\cZ_{r,s}(A,B;p)-  \frac{2 A \# \cZ_s(B,p) }{p-1}\right|^2\\
& \le &  p^{-\vartheta+o(1)}
 \sum_{r,s\in \F_p^*}
\left|\#\cZ_{r,s}(A,B;p)-  \frac{2 A \# \cZ_s(B,p) }{p-1}\right|^2. 
\end{eqnarray*}
Now, using Lemma~\ref{lem:ZrsAB} for each $s \in \F_p^*$,  we obtain
$$
\Delta^2 \ll  \(\frac{A}{p}+1\)^2 \(\frac{B}{p} + 1\) B  p^{2-\vartheta+o(1)}.
$$
Hence, 
$$
\Delta \le 
ABp^{-1/2 -\vartheta/2 + o(1)} + AB^{1/2} p^{-\vartheta/2 + o(1)}
+Bp^{1/2 -\vartheta/2 + o(1)} +
B^{1/2} p^{1-\vartheta/2 + o(1)}. 
$$
Recalling~\eqref{eq:M Delta} and using that  
$$
ABp^{-1/2 -\vartheta/2}>ABp^{-1/2 -\vartheta}, 
$$
we obtain 
$$
\left| M_p(\cS,A,B)    - \frac{ 4 A B \#\cS  }{(p-1)^2}  \right|
\le E_\vartheta(A,B;p) p^{o(1)}.
$$
It is also easy to see that the  roles of $A$ and  $B$ can  be interchanged
in all previous arguments which  
concludes the proof. 
\end{proof}

In all our applications in this paper we have $\vartheta \le 1/2$, 
in which case $UVz^{-1/2 -\vartheta/2} \ge  UVz^{-1}$ and
$U z^{1/2 - \vartheta} \ge U$, thus we 
obtain
\begin{equation}
\label{eq:simpl E}
\begin{split}
E_\vartheta(U,V;z) \ll  UVz^{-1/2 -\vartheta/2 }  +  
UV^{1/2} z^{-\vartheta/2} + &
U z^{1/2 - \vartheta} \\
  +&   Vz^{1/2 -\vartheta/2}    + V^{1/2} z^{1-\vartheta/2}.
\end{split}
\end{equation}

\subsection{The Lang--Trotter Conjecture on Average}

We now derive the following generalisations 
of~\eqref{eq:L-T Aver}.

\begin{theorem}
\label{thm:LT Aver} Suppose 
that $f(T), g(T) \in \Z[T]$ are not powers of another polynomial 
over $\Q$.  Assume that   positive integers $A$ and $B$  
are such that for some $\eps> 0$ we have
$$
\max\{AB^{1/2},A^{1/2}B\} \ge x^{5/4 +\eps}
  \mand  \min\{A,B\} \ge x^{1/2 +\eps}.
$$
Then we have
$$
\frac{1}{4AB}\sum_{|a|\le A}\sum_{|b|\le B}
\Pi^{\tt LT}_{f(a),g(b)}(t;x)  =   \(C(t)+  o(1) \)\frac{\sqrt x}{\log
x}.
$$
for some constant $C(t)>0$ depending only on $t$.
\end{theorem}

\begin{proof} In the notations of Sections~\ref{sec:Isom} 
and~\ref{sec:Ell Curves}, we have 
$$
\sum_{|a|\le A}\sum_{|b|\le B}
\Pi^{\tt LT}_{f(a),g(b)}(t;x)  = \sum_{p \le x} M_p(\cR_p(t),A,B).
$$

Assume that $A \ge B$. 
Now from Theorem~\ref{thm:M and S} applied with $\cS = \cR_p(t)$
(thus $\vartheta= 1/2$ by Lemma~\ref{lem:LT Upper},  so we can also use~\eqref{eq:simpl E}),  we derive
\begin{equation*}
\begin{split}
\sum_{|a|\le A}\sum_{|b| \le B} &
\Pi^{\tt LT}_{f(a),g(b)}(t;x)  -  4 A   B  
\sum_{p \le x}   \frac{\cR_p(t)}{(p-1)^2}   \\
& \le    \(ABx^{1/4}  +    A B^{1/2} x^{3/4}  
+ A x + Bx^{5/4} + B^{1/2} x^{7/4}\)x^{o(1)} , 
\end{split}
\end{equation*} 
as $x\to \infty$. 

By Lemma~\ref{lem:LT Upper} and Lemma~\ref{lem:LT Stat}
(taken with, say, $C=2$)
we see that 
\begin{eqnarray*}
\sum_{p \le x} \frac{ \#\cR_p(t)  }{(p-1)^2} &=& \sum_{p \le x} 
 \#\cR_p(t) \(\frac{1}{p^2} + O\(p^{-3}\)\) \\
&=&   \sum_{p \le x} \frac{1}{p^2} \# \cR_p(t)   + O(1)\\
&=& C(t)\li_{1/2}(x) + O\(x^{1/2} (\log x)^{-2}\) .
\end{eqnarray*}
Hence (disregarding  the term $A B  x^{1/4 }$,  which is obviously smaller 
than the contribution from the error term in the previous 
formula), we obtain
\begin{equation*}
\begin{split}
\sum_{|a|\le A}&\sum_{|b| \le B} 
\Pi^{\tt LT}_{f(a),g(b)}(t;x)     -  4    A B  C(t)\li_{1/2}(x)  \\
  \ll  &  \( A B^{1/2} x^{3/4 }  
   +   A x   +     Bx^{5/4 } + B^{1/2} x^{7/4}\)x^{o(1)}+ AB x^{1/2} (\log x)^{-2} 
    \\
    &\ll  AB x^{1/2+o(1)} \(B^{-1/2}  x^{1/4} 
  +  B^{-1} x^{1/2}  +  A^{-1} x^{3/4} + A^{-1}B^{-1/2} x^{5/4}\)  \\
  &\qquad \qquad \qquad \qquad \qquad\qquad\qquad 
  \qquad \qquad\qquad \quad  +  AB x^{1/2} (\log x)^{-2} .
\end{split}
\end{equation*}
Since $A \ge B$, 
 we have $A^{3/2} \ge AB^{1/2} \ge x^{5/4+ \eps}$.
 Thus $A \ge x^{5/6+ 2\eps/3}$, 
and  after simple calculations, we obtain the  desired result in the case  $A \ge B$.

In the case $A < B$ the proof is completely analogous. 
\end{proof}

Since $\max\{AB^{1/2},A^{1/2}B\} \ge (AB)^{3/4}$
we can replace the first condition of Theorem~\ref{thm:LT Aver}
with $AB  \ge x^{5/3+\eps}$.

\subsection{The Sato--Tate Conjecture on Average}

We now use Theorem~\ref{thm:M and S} to obtain 
a generalisation of~\eqref{eq:S-T Aver} and~\eqref{eq:BS threshold}.

We  recall that the Sato--Tate density 
$\mu_{\tt ST}(\alpha,\beta)$ is given by~\eqref{eq:ST dens}. 

\begin{theorem}
\label{thm:ST Aver}  Suppose 
that $f(T), g(T) \in \Z[T]$ are not powers of another polynomial 
over $\Q$.
Assume that   positive integers $A$ and $B$  
are such that for some $\eps> 0$ we have
$$
\max\{AB^{1/2},A^{1/2}B\} \ge x^{1 +\eps}
  \mand  \min\{A,B\} \ge x^{1/2 +\eps}.
$$
Then for  all real numbers $0 \le \alpha < \beta \le \pi$, we have
$$
\frac{1}{4AB} \sum_{|a|\le A}\sum_{|b| \le B}
\Pi^{\tt ST}_{f(a),g(b)}(\alpha,\beta;x)= \(\mu_{\tt ST}(\alpha,
\beta) + O\(x^{-\delta}\)\) \pi(x),
$$
where  $\delta > 0$ depends only on  $\eps$.
\end{theorem}

\begin{proof}In the notations of Sections~\ref{sec:Isom} 
and~\ref{sec:Ell Curves}, we have 
$$
\sum_{|a|\le A}\sum_{|b| \le B}
\Pi^{\tt ST}_{f(a),g(b)}(\alpha,\beta;x)  = \sum_{p \le x} M_p(\cT_p(\alpha,\beta),A,B).
$$

Assume that $A \ge B$. 
Now from Theorem~\ref{thm:M and S} applied with $\cS = \cT_p(\alpha,\beta)$
(thus $\vartheta= 0$ so we can also use~\eqref{eq:simpl E}),  we derive
\begin{equation*}
\begin{split}
\sum_{|a|\le A}\sum_{|b| \le B} 
\Pi^{\tt ST}_{f(a),g(b)}&(\alpha,\beta;x)    -  4 A   B  
\sum_{p \le x}   \frac{\cT_p(\alpha,\beta) }{(p-1)^2}  \\
 & \le    \(ABx^{1/2}  +    A B^{1/2} x  
   +  A x^{3/2}      +Bx^{3/2} + B^{1/2} x^{2}\)x^{o(1)}, 
\end{split}
\end{equation*} 
as $x\to \infty$. 
Since $A\ge B$ this simplifies as 
\begin{equation*}
\begin{split}
\sum_{|a|\le A} \sum_{|b| \le B} 
\Pi^{\tt ST}_{f(a),g(b)}&(\alpha,\beta;x)     -  4 A   B  
\sum_{p \le x}   \frac{\cT_p(\alpha,\beta) }{(p-1)^2}  \\
  &\le     \(ABx^{1/2}  +    A B^{1/2} x  
   +  A x^{3/2}  + B^{1/2} x^{2}\)x^{o(1)} . 
\end{split}
\end{equation*} 
Using Lemma~\ref{lem:ST Stat}
(which contributes  $AB x^{3/4 + o(1)}$ to the error term
and thus dominates the term $ ABx^{1/2  + o(1)}$) 
we obtain 
\begin{equation*}
\begin{split}
\sum_{|a|\le A}\sum_{|b| \le B} &
\Pi^{\tt ST}_{f(a),g(b)}(\alpha,\beta;x)     -  4   \mu_{\tt ST}(\alpha,\beta) A B \pi(x) \\
  \le & \(AB x^{3/4 }  +    AB^{1/2} x  +   A x^{3/2} + B^{1/2} x^{2}\) x^{o(1)} \\
  & = ABx^{1+o(1)} \( x^{-1/4 }  +    B^{-1/2}   +  B^{-1} x^{1/2} + A^{-1}B^{-1/2} x\) 
\end{split}
\end{equation*}
and  after simple calculations, we obtain the  desired result
in the case $A \ge B$.

In the case $A < B$ the proof is completely analogous. 
\end{proof}

Since $\max\{AB^{1/2},A^{1/2}B\} \ge (AB)^{3/4}$
we can replace the first condition of Theorem~\ref{thm:ST Aver}
with $AB  \ge x^{4/3+\eps}$.

Finally, from Lemma~\ref{lem:ST Stat Poly} we instantly obtain 
a result for one parametric families of elliptic curves.

\begin{theorem}
\label{thm:ST Aver Onepar}  Suppose 
that $f(T), g(T) \in \Z[T]$ are such that
$4f(T)^3 +27g(T)^2$ is not identical 
to zero.  Assume that a  positive integer $A$  
is  such that for some $\eps> 0$ we have
$$
 A  \ge x^{1/2 +\eps}.
$$
Then for  all real numbers $0 \le \alpha < \beta \le \pi$, we have
$$
\frac{1}{2A} \sum_{|a|\le A} 
\Pi^{\tt ST}_{f(a),g(a)}(\alpha,\beta;x)= \(\mu_{\tt ST}(\alpha,
\beta) + O\(x^{-\delta}\)\) \pi(x),
$$
where  $\delta > 0$  depends only on  $\eps$.
\end{theorem}

\section{Comments}

We remark that  Theorem~\ref{thm:ST Aver Onepar}
may seem to imply  Theorem~\ref{thm:ST Aver} (that is, for every
$b$ with $|b|\le B$ one can try to apply Theorem~\ref{thm:ST Aver Onepar}
to the corresponding family of curves). 
However, this is not the case as due to the uniformity 
issue with respect to $b$.
We also note that Theorem~\ref{thm:ST Aver} is just
an example of several similar results that hold
under the same conditions on $A$ and $B$ and describe 
the distribution of curves with special properties. 
As in~\cite{BaCoDa,BaSh}, these properties may include
cyclicity, primality or divisibility of $\# E_{f(a),g(b)}(\F_p)$ 
by a given  integer and several others. On the other hand,
it is not clear how to obtain 
analogues of  Theorem~\ref{thm:ST Aver Onepar} for such 
questions. 

%


 
\section*{Acknowledgements} 

The author would   like to thank  Stephan Baier  for 
very interesting discussions and the observation that 
the argument of his work~\cite{Baier2} may allow 
us to obtain a version of  Theorems~\ref{thm:LT Aver}
in a very large range of uniformity with respect 
to the parameter $t$. 

This work was supported in part by ARC grant DP0881473.

\end{document}